\documentclass[a4paper,12pt,oneside,reqno]{amsart}
\usepackage{hyperref}
\usepackage[latin1]{inputenc}

\usepackage[headinclude,DIV13]{typearea}
\areaset{15.1cm}{25.0cm}
\usepackage{amsfonts,amssymb,amsmath,bbm}
\usepackage{amsthm}
\usepackage{cancel} 
\usepackage{graphicx}
\usepackage{caption}
\usepackage{subcaption}

\usepackage{xcolor}
\usepackage{enumerate}
\usepackage{dsfont}
\usepackage{nicefrac}

\theoremstyle{plain}
\newtheorem{theorem}{Theorem}
\newtheorem{lemma}[theorem]{Lemma}

\theoremstyle{definition}

\newtheorem{assumption}[theorem]{Assumption}

{\itshape}{\rmfamily}

\theoremstyle{remark}
\newtheorem{remark}{Remark}

\def\paragraph#1{\noindent \textbf{#1}}

\numberwithin{equation}{section}

\def\Cov{\mathop{\rm Cov}\nolimits}

\def\d{\mathrm{d}}

\def\<{\langle}
\def\>{\rangle}
\def\a{\alpha}

\def\e{\e}

\def\l{\lambda}
\def\r{\rho}
\def\s{\sigma}

\def\th{\theta}
\def\vt{\vartheta}

\def\z{\zeta}
\def\o{\omega}

\def\L{\Lambda}

\def\O{\Omega}

\def\R{{\Bbb R}}  
\def\N{{\Bbb N}}  
\def\P{{\Bbb P}}  
\def\Z{{\Bbb Z}}  
\def\E{{\Bbb E}}  

\def\Va{{\Bbb V}}  

\let\cal=\mathcal

\def\BB{{\cal B}}
\def\CC{{\cal C}}
\def\DD{{\cal D}}

\def\FF{{\cal F}}

\def\OO{{\cal O}}

\def\WW{{\cal W}}

\def\YY{{\cal Y}}
\def\ZZ{{\cal Z}}

 \def \L {{\Lambda}}

 \def \e {{\varepsilon}}
 \def \s {{\sigma}}
 
 \def \z {{\zeta}}

 \def \th {{\theta}}

 \def \l {{\lambda}}
 \def \d {{\delta}}
 \def \a {{\alpha}}
 \def \o {{\omega}}
 \def \O {{\Omega}}
 \def \th {{\theta}}

 \def \r {{\rho}}

 \def \ba {\begin{array}}
 \def \ea {\end{array}}

 \newcommand{\be}{\begin{equation}}
 \newcommand{\ee}{\end{equation}}

\newcommand{\bea}{\begin{eqnarray}}
 \newcommand{\eea}{\end{eqnarray}}
 
\def\TH(#1){\label{#1}}\def\thv(#1){\ref{#1}}
\def\Eq(#1){\label{#1}}\def\eqv(#1){(\ref{#1})}

\def\sfrac#1#2{{\textstyle{#1\over #2}}}

 \def \1{\mathbbm{1}}
\def\wt {\widetilde}
\def\wh{\widehat}

\def \so {\text{\small{o}}(1)}
\def \po {\text{\small{o}}}

\begin{document}

 \title[Conditional  Strong Large Deviations]{A conditional strong large deviation result and a functional central limit theorem for the rate function}
 
\author[A. Bovier]{Anton Bovier}
 \address{A. Bovier\\Institut f\"ur Angewandte Mathematik\\
Rheinische Friedrich-Wilhelms-Universit\"at\\ Endenicher Allee 60\\ 53115 Bonn, Germany}
\email{bovier@uni-bonn.de}
\author[H. Mayer]{Hannah Mayer}
\address{H. Mayer\\ Institut f\"ur Angewandte Mathematik\\
Rheinische Friedrich-Wilhelms-\newline Universit\"at\\ Endenicher Allee 60\\ 53115 Bonn, Germany}
\email{hannah.mayer@uni-bonn.de}


\thanks{We are very grateful to Frank den Hollander for numerous discussions and valuable input. A.B. is partially supported through the German Research Foundation in the Priority Programme 1590 ``Probabilistic Structures in Evolution'' and the Hausdorff Center for Mathematics (HCM).  He is member of the Cluster of Excellence ``ImmunoSensation'' at Bonn University.
H.M. is supported by the German Research Foundation in the Bonn International Graduate School in Mathematics (BIGS) and the Cluster of Excellence ``ImmunoSensation''.}

\begin{abstract}
We study the large deviation behaviour of $S_n=\sum_{j=1}^nW_jZ_j$, where 
$(W_j)_{j\in \N}$ and $(Z_j)_{j\in\N }	$ are sequences of real-valued, independent and 
identically distributed random variables satisfying certain moment conditions, 
independent of each other. More precisely, we prove a conditional strong large deviation 
result and describe the fluctuations of the 
random rate function through a  functional central limit theorem. 
\end{abstract}

\maketitle

\section{Introduction and Results}
Let  $(Z_j)_{j\in\N}$ be  independent, identically distributed (i.i.d.) random variables and let $(W_j)_{j\in \N}$ be i.i.d. random variables as well. Define the $\sigma$-fields \mbox{$\ZZ \equiv \sigma(Z_j, j\in\N)$} and \mbox{$\WW \equiv \sigma(W_j, j\in \N)$} and let $\ZZ$ and $\WW$ be independent. Furthermore, define 
\begin{equation}\Eq(first.1)
	S_n \equiv \sum_{j=1}^n Z_jW_j.
\end{equation}
In this paper we derive \emph{strong (local) large deviation} estimates on  $S_n$ 
\emph{conditioned} on the $\sigma$-
field $\WW$. The random variables $W_j$ can be interpreted as a random environment 
weighting the summands of $S_n$. Conditioning on $\WW$ can thus be understood as 
fixing the environment.  Comets \cite{C} investigates conditional large deviation estimates of 
such sums in the more 
 general setup of i.i.d. random fields of random variables taking values in a Polish 
 Space.  His results concern, however, only the standard \emph{rough} large deviation 
 estimates. Local limit theorems have been obtained in the case $S_n\in \R$  
 (see e.g. \cite{BR, CS}) and for the case $S_n\in\R^d$ (see \cite{I}), 
 but these have, to our knowledge,
  not been applied to conditional laws of sums of the form \eqv(first.1). 

Our result consists of two parts. The first part is an almost sure local limit theorem 
for the conditional tail probabilities $\P(S_n\geq a n|\WW)$,  $a\in \R$. 
The second part is a functional central limit theorem for the \emph{random} rate 
function. 

\subsection{Strong large deviations}
For a general review of \emph{large deviation theory} see for example den Hollander \cite{FdH} or 
Dembo and Zeitouni \cite{DZ}.
A \emph{Large deviation  principle} for a family of real-valued random variables $S_n$ 
roughly says that, for $a>\E \left[\frac 1nS_n\right]$,
\begin{equation}
	\P(S_n\geq an)=\exp\left[-nI(a)(1 + \so)\right].
\end{equation}
The G\"artner-Ellis theorem asserts that the 
 \emph{rate function}, $I(a)$, is obtained as the limit of 
the Fenchel-Legendre transformation of the logarithmic moment generating function of $S_n$, to wit $I(a)=\lim_{n\to\infty} I_n(a)$, where $I_n(a)$ is defined by
\begin{equation}
	 I_n(a)\equiv \sup_{\vt}(a\vt-\Psi_n(\vt)) = a\vt_n-\Psi_n(\vt_n),
\end{equation}
where $\Psi_n(\vt)\equiv \frac 1n\log\E[\exp(\vt S_n)]$ and $\vt_n$ satisfies $\Psi_n'(\vt_n)=a$.

 \emph{Strong large deviations estimates }  refine this  \emph{exponential  asymptotics}.
They provide estimates of the form  
\begin{equation}
	\P(S_n\geq an)=\frac{\exp(-nI_n(a))}{\vt_n\s_n\sqrt{2\pi n}}[1+\so],
\end{equation}
where $\s_n^2\equiv \Psi_n''(\vt_n)$ denotes the variance of $\frac{1}{\sqrt n}S_n$
under the \emph{tilted} law $\wt \P$ that has density 
\begin{equation}\Eq(tilted.1)
\frac {d\wt\P}{d\P} =\frac{\hbox{\rm e}^{\vt_nS_n}}{\E \left[\hbox{\rm e}^{\vt_nS_n}\right]}.
\end{equation}
 The standard theorem for $S_n$ a sum of i.i.d. random variables is due to Bahadur and Ranga Rao \cite{BR}. The generalisation, which we summarise by Theorem \ref{abwthm},
is a result of  Chaganty and Sethuraman \cite{CS}. We abusively refer to $I_n(a)$ as the \emph{rate function}.
%

The following theorem is based on 2 assumptions. 
\begin{assumption}\label{ass1}
	There exist $\vt_* \in (0,\infty)$ and $\beta < \infty$ such that
	\begin{equation}
	|\Psi_n(\vartheta)|<\beta, \text{ for all } \vartheta \in  \lbrace \vartheta \in \mathbb{C} : |\vartheta| <\vartheta_*\rbrace  .
	\end{equation} 
	for all $ n\in\mathbb N$ large enough.
\end{assumption}	
\begin{assumption}\label{ass2}
$(a_n)_{n\in\mathbb{N}}$ is a bounded real-valued sequence such that  the equation 
	    \begin{equation} \label{Theta}
		  a_n = \Psi_n'(\vartheta)
	    \end{equation}
	has a solution
 	$\vartheta_n \in (0,\vartheta_{**})$ with $\vartheta_{**} \in (0,\vartheta_*) $ for all $n\in\mathbb{N}$ large enough.
\end{assumption}
\begin{theorem}[Chaganty and Sethuraman \cite{CS}] \label{abwthm}
Let $S_n$ be a sequence of real-valued random variables defined on a probability space $(\O,\FF,\P)$. Let $\Psi_n$ be their logarithmic moment generating function defined above and assume that
	 Assumptions \ref{ass1} and \ref{ass2} hold for $\Psi_n$. 
	 Assume furthermore that
 \begin{enumerate}[(i)]
    \item $\lim_{n \to \infty} \vartheta_n \sqrt{n} = \infty $ \label{1}
    \item $\liminf_{n\to \infty} \sigma_n^2 > 0$ and \label{2}
    \item $\lim_{n \to \infty} \sqrt{n} \sup_{\delta_1 \leq \left|t\right|\leq\delta_2\vt_n} 	\left|\frac{\Phi_n\left(\vt_n+it\right)}{\Phi_n\left(\vt_n\right)}\right|=0 \quad \forall \, 0<\delta_1<\delta_2<\infty$, \label{3}
  \end{enumerate}
	are satisfied. 
	Then
\begin{equation}\label{approx}
				\mathbb{P}\left(S_n\geq na_n\right) = \frac{\exp(-nI_n(a_n))}{\vt_n\sigma_n\sqrt{2\pi n}} \left[1+\so \right], \, n\to \infty .
\end{equation}
\end{theorem}
This result is deduced from a local central limit theorem for $\frac{{S}_n-na_n}
{\sqrt {n\s_n^2}}$ under the tilted law $\wt \P$ defined in \eqv(tilted.1).
\begin{remark}
There are estimates for $\P(S_n\in n\Gamma)$, where $S_n \in \R^d$ and $\Gamma \subset \R^d$, see \cite{I}. Then the leading order prefactor depends on $d$ and the geometry of the set $\Gamma$.  
\end{remark}
\subsection{Application to the conditional scenario}
Throughout the following we write $I_n^{\WW}(a)$ and $\vt_n^{\WW}(a)$ to emphasise that these are random quantities.

\begin{remark}
Alternatively, one could condition on a different $\sigma$-field $\YY$ as in the application to financial mathematics and an immunological model described in Section \ref{applications}. In the proofs we just need the fact that $\WW \subset \YY$ and $\ZZ$ is independent of $\YY$. 
\end{remark}

\begin{theorem}\label{LargeDev} Let $S_n$ be defined in \eqv(first.1). 
Assume that the random variables $W_1$ and $Z_1$ satisfy the following conditions:
\begin{enumerate}[(i)]
	\item  {$Z_1$ is not concentrated on one point.}
	\begin{enumerate}
		\item If $Z_1$ is lattice valued, $W_1$ has an absolutely continuous part and there exists an interval $[c,d]$ such that the density of $W_1$ on $[c,d]$ is bounded from below by $p>0$.  \label{Nolattice}
		\item  If $Z_1$ has a density, $\P(|W_1|>0)>0$ \label{Nolattice2}.
		\end{enumerate}
	\item  \label{momgenZ} The moment generating function of $Z_1$, $M(\vt)\equiv \mathbb E[\exp(\vartheta Z_1)]$, is finite for all $\vartheta\in \R$. %
	\item  For $f(\vt)\equiv \log M (\vt)$, both $\E[f(\vt W_1)]$ and $\E[W_1f'(\vt W_1)]$ are finite for all $\vt\in \R$.\label{expectations}
	\item There exists a function $F:\R\to\R$ such that $\E[F(W_1)]<\infty$ and \\
	 $W_1^2f''(\vt W_1)\leq F(W_1)$. \label{majorant}
\end{enumerate}

Let $\vt_*\in \R_+$ be arbitrary but fixed. Let  $J\equiv\left(\E[W_1]\E[Z_1], \E[W_1f'(\vt_{*}W_1)]\right)$ and let  $a\in J$.
Then 
\begin{equation}\label{representation}
\P\Bigg( \forall a \in J: \P(S_n\geq an|\WW)=\frac{\exp(-nI_n^{\WW}(a))}{\sqrt{2\pi n}\vt_n^{\WW}(a)\s_n^{\WW}(a)}(1+\so)\Bigg)=1,
\end{equation}
where
 \begin{equation}\label{ratefunction}
 	I_n^{\WW}(a) =   a\vartheta_n^{\WW}(a)-\frac 1n \sum_{j=1}^{n}f\left(W_j\vt_n^{\WW}(a)\right)
 \end{equation}
and $\vt_n^{\WW}(a)$ solves $a=  \frac{d}{d\vt}(\frac 1n \sum_{j=1}^n  f(W_j \vt))$.
\end{theorem}


This theorem is proven in Section \ref{proofs1}.

\begin{remark}
The precise requirements on the distribution of  $W_1$ depend on the distribution of 
$Z_1$. In particular, Condition \eqref{expectations} does not in general require the moment 
generating function of $W_1$ to be finite for all $\vt\in \R$. Condition \eqref{majorant} 
looks technical. It is used to establish Condition \eqref{2} of Theorem \ref{abwthm} for all 
$a$ at the same time. For most applications, it is not very restrictive, see Section 
\ref{Satisfaction} for examples. 
\end{remark}

\subsection{Functional central limit theorem for the random rate function}

Note that the rate function  $I_n^{\WW}(a)$ is  \emph{random}. Even if we may expect 
that $I_n^{\WW}(a) \rightarrow \E\left[I^\WW_n(a)\right]$, almost surely, due to the fact 
that it is multiplied by $n$ in the exponent in  Equation \eqv(representation), its 
 fluctuations are relevant. To control them, we prove  a functional central limit theorem. 
 We introduce the following notation. 
\begin{equation}\Eq(plop.30)
 g(\vt)\equiv \E[f(W_1\vt)] \quad  \text{ and  } \quad X_n(\vt)\equiv \frac{1}{\sqrt n}\sum\limits_{j=1}^{n}\left( f(W_j\vt)-\E [f(W_j\vt)]\right).
\end{equation}
Moreover, define 
$\vt(a)$ as the solution of the equation $a=g'(\vt)$.

In addition to the assumptions made in Theorem \thv(LargeDev), we need the following  assumption on the covariance structure of the summands appearing in the definition of
$X_n(\vt)$ and their derivatives.

\begin{assumption} \label{Cov} There exists $C<\infty$, such that,
  for all $a, a' \in \bar J$, where $\bar J$ is the closure of the interval $J$,
\begin{align*} &\Cov\left(f(\vt(a)W_j),f(\vt(a')W_j)\right), \quad\Cov\left(W_jf'(\vt(a)W_j),W_jf'(\vt(a')W_j)\right), \nonumber \\
 &\Cov\left(W_j^2f''(\vt(a)W_j),W_j^2f''(\vt(a')W_j)\right), \quad \Cov\left(f(\vt(a)W_j),W_jf'(\vt(a')W_j)\right), \nonumber \\
 &\Cov\left(W_jf'(\vt(a)W_j),W_j^2f''(\vt(a')W_j)\right), \quad \Cov\left(f(\vt(a)W_j),W_j^2f''(\vt(a')W_j)\right) \text{ and } \nonumber \\
&\Va\left[W_j^3f'''(\vt(a)W_j)\right]
\end{align*}
are all smaller than $C$.
\end{assumption}

  \begin{theorem}\label{FCLT}
If  $g''(\vt(a))>c$ for some $c>0$ and Assumption \ref{Cov} is satisfied, then the rate function satisfies
 \begin{equation}\label{convrescaled}
 I^\WW_n(a)=I(a)+ n^{-1/2} X_n(\vt(a)) +n^{-1} r_n(a),
 \ee
 where 
 \be\Eq(det.1)
 I(a)\equiv a\vt(a)-g(\vt(a)),
 \ee
 \be\Eq(fluct.1)
 (X_n(\vartheta(a)))_{a\in \bar J}\,{\stackrel {\DD} {\to}}\,(X_a)_{a\in \bar J},\quad \hbox{\rm as}
 \, n\to \infty,
 \ee
 where $X$ is the Gaussian process with mean zero and covariance
 \begin{equation}
\Cov (X_a,X_{a'}) =  \E[f(W_1\vt(a))f(W_1\vt(a'))]-\E[f(W_1\vt(a))]\E[f(W_1\vt(a'))],
\end{equation}
and 
\begin{equation}
r_n(a) = \frac{ (X_n'(\vt(a)))^2}{2\left[g''(\vt(a))+\frac 1{\sqrt n}X_n''(\vt(a))\right]}+\po\left( 1\right),
\end{equation} 
 uniformly in $a\in \bar J$.
\end{theorem}%

To prove Theorem \ref{FCLT} we show actually more, namely that  the process
\be
(X_n(\vt(a)), X_n'(\vt(a)), X_n''(\vt(a)))_{a\in \bar J} \,{\stackrel {\DD} {\to}}\,
 (X_a,X^{'}_a,X^{''}_a)_{a\in \bar J},
 \ee
  (see Lemma \ref{JWC} below). The proof of the theorem is given in Section \ref{proofs2}.

\subsection{Examples}\label{Satisfaction}
In the following we list some examples in which the conditions of the preceding 
theorems are satisfied.
\begin{enumerate}
\item Let  $Z_1$ be a Gaussian random variable
 with mean zero and variance $\s^2$. In this case, \begin{eqnarray}
	f(\vt)&=&\log(\E[\exp(\vt Z_1)])=\frac 12 \s^2\vt^2, \quad f'(\vt)=\s^2\vt \\
	f''(\vt)&=& \s^2, \quad \text{ and } \quad f'''(\vt)=0
\end{eqnarray}
This implies that  $W_1$ must  have finite fourth moments to satisfy Assumption 
\ref{Cov}. Under this requirement Conditions \eqref{expectations} and \eqref{majorant} 
of Theorem \ref{LargeDev} are met. According to Condition \eqref{Nolattice2} of 
Theorem \ref{LargeDev}, $W_1$ may not be concentrated at $0$. Moreover,
\begin{equation}
	g''(\vt)=\s^2 > c
\end{equation}
independent of the distribution of $W_1$.
\item Let $Z_1$ be a binomially distributed random variable, $Z_1 \sim B(m, p)$. Thus
\begin{eqnarray}
	f(\vt)&=& m \log(1-p+pe^{\vt})\\
	f'(\vt)&=&m\frac{p e^{\vt}}{1-p+pe{\vt}}\leq m \\
	f''(\vt)&=&m(p-p^2)\frac{e^{\vt}}{(1-p+pe^{\vt})^2} \leq f''\left(\log\left(\frac{3p-1}{p}
	\right)\right) \\
	f'''(\vt)&=&m(p-p^2)e^{\vt}\frac{1-3p+pe^{\vt}}{(1-p+pe^{\vt})^3} \in C_0.
\end{eqnarray}
Then $W_1$ has to satisfy \eqref{Nolattice} of Theorem \ref{LargeDev} and must
 have finite sixth moments. 
One can show that $f'(\vt), f''(\vt)$ and  $f'''(\vt)$ are bounded, $\E[f(\vt W_1)]$ and the 
moments depending on $f(\vt W_1)$ in Assumption \ref{Cov} are finite.
Furthermore, the assumption $0<\E[W_1^2]<\infty$ implies that $g(\vt(a))>c$ as 
required in Theorem \ref{FCLT}.
\end{enumerate}

\begin{remark}
In both cases it is not necessary that the moment generating function of $W_1$ exists.
\end{remark}

\subsection{Related results} 

After posting our manuscript on arXiv, Ioannis Kontoyiannis informed us about the
 papers
 \cite{DemKon1999} and \cite{DemKon2002}  by Dembo and Kontoyiannis, where
 some similar results on conditional large deviations are obtained. 
 They concern sums of the form 
\be\label{DK3}
\r_n\equiv \frac 1n \sum_{j=1}^n\r(W_j,Z_j), 		
\ee
where   ${\bf W}=(W_j)_{j\in\N}$ and ${\bf Z} =(Z_j)_{j\in\N}$ are two stationary 
processes with $W_j$ and $Z_j$ taking values in some Polish spaces $A_W$ and
 $A_Y$, respectively, and $\r: A_W\times A_Z \to [0,\infty)$ is some 
 measurable function. Their main motivation is to estimate 
 the frequency with which subsequences of length $n$ in the process $Z$ occur that 
 are "close" to $W$.  To do this, they estimate conditional probabilities of the form 
 \be
 \P\left(\r_n\leq D|\WW\right),
\ee
obtaining, under suitable assumptions,  refined large deviation estimates of the form 
\be
\frac 1n \ln   \P\left(\r_n\leq D|\WW\right) = R_n(D) + \frac 1{\sqrt n} \L_n(D) +o(1/\sqrt n),
\ee 
almost surely, 
where they show that $R_n(D) $ converges a.s. while $\L_n(D)$ converges in 
distribution to a Gaussian random variable.

\section{Applications}\label{applications}
\subsection{Stochastic model of T-cell activation} 
The  immune system  defends the body against dangerous intrusion, e.g. bacteria, 
viruses and cancer cells. The interaction of so-called T-cells and antigen presenting cells 
plays an important r\^ole in performing this task. Van den Berg, Rand and Burroughs 
developed a stochastic model of T-cell activation in \cite{BRB} which was further 
investigated by Zint, Baake and den Hollander in \cite{BZH} and Mayer and Bovier in  \cite{MB}. 
Let us briefly explain this model. 

The antigen presenting cells display on their surface a mixture of peptides present in the 
body. During a bond between a T-cell and a presenting cell the T-cell scans the 
presented mixture of peptides. The  T-cell is stimulated during this process, and if the 
sum of all stimuli exceeds a threshold value, the cell becomes activated and triggers
 an immune response. The signal received by the T-cell is represented by
\begin{equation}
	S_n\equiv \sum_{j=1}^nZ_jW_j+z_{\mathsf f}W_{\mathsf f},
\end{equation}
where $W_j$ represents the stimulation rate elicited by a peptide of type $j$ and $Z_j$ 
represents the random number of presented peptides of type $j$. The sum describes 
the signal due to self peptides, $z_{\mathsf f}W_{\mathsf f}$ is the signal due to one 
foreign peptide type. From the biological point of view, T-cell activations are rare events 
and thus large deviation theory is called for 
to investigate $\P(S_n\geq na|\YY)$, where $\YY$ is a $\sigma$-field such 
 that $W_j$ are measurable with respect to $\YY$ and $Z_j$ are independent of $\YY$. 
 For two examples of distributions discussed in \cite{BZH},  Theorems \ref{LargeDev}  
 and \ref{FCLT} can be applied. In both examples, the 
 random variables $Z_j$ are  binomially distributed,  and thus their 
 moment generating function exists everywhere. $W_j$ is defined by $W_j \equiv \frac 
 {1}{\tau_j}\exp(-\frac{1}{\tau_j})$, where  $\tau_j$ are exponentially distributed or 
 logarithmic normally distributed, i.e. $W_j$ are bounded and the required moments 
 exist. Furthermore, $W_1$ has a density and Condition \eqref{Nolattice} of Theorem 
 \ref{LargeDev} is met. Using Theorems \ref{LargeDev} and \ref{FCLT}, one can prove 
 that the probability of T-cell activation for a given type of T-cell grows exponentially with 
 the number of presented foreign peptides, $z_{\mathsf f}$, if the corresponding 
 stimulation rate $W_{\mathsf f}$ is sufficiently large. It is then argued that a suitable 
 activation threshold can be set that allows significantly differentiate between the 
 presence or absence of foreign peptides. 
For more details see  \cite{MB}.

\subsection{Large portfolio losses} Dembo, Deuschel, and Duffie investigate in  
\cite{DDD} the probability of large financial losses on a bank portfolio or the total claims 
against an insurer conditioned on a macro environment. The random variable $S_n$ 
represents the total loss on a portfolio consisting of many positions, $W_j$ is a $\{0,1\}$-valued random variable and indicates if 
position $j$ experiences a loss, whereas the random variable $Z_j$ is for example exponentially 
distributed and represents the amount of loss. They consider the probability conditioned 
on a common macro environment $\YY$ and assume that $Z_1, W_1, \dots, Z_n, W_n$ 
are conditionally independent. Furthermore, they work in the slightly generalised setup 
of finitely many blocks of different distributions. That is
\begin{equation}
S_n\equiv\sum_{\alpha=1}^K\sum_{j=1}^{Q_{\alpha}}Z_{\alpha, j}W_{\alpha, j},
\end{equation}
where $Z_{\alpha,j}\stackrel{\DD}{=}Z_{\alpha}$ and $W_{\alpha, j}\stackrel{\DD}{=}
W_{\alpha}$ for each $\alpha \in \{1,\dots,K\}$ and $\sum_{\alpha=1}^K Q_{\a}=n$. Moreover, the conditional probability of losses for each position is calculated and the influence of the length of the time interval, in which the loss occurs, is investigated. For more details see \cite{DDD}.

\begin{remark}
In general, the exponential distribution for $Z_1$ causes problems because the moment 
generating function does not exist everywhere. Evaluating at $\vt W_j$ thus might yield 
to an infinite term depending on the range of $W_j$. In this application there is no 
problem because $W_j$ is $\{0,1\}$-valued. 
\end{remark}

\section{Proof of Theorem \thv(LargeDev)}\label{proofs1}

\begin{proof}[Proof of Theorem \ref{LargeDev}]
We prove Theorem \thv(LargeDev) by showing that the conditional law 
of $S_n$ given $\WW$ satisfies the assumptions of Theorem 
\thv(abwthm) uniformly in $a\in J$, almost surely.

Assumption \ref{ass1} is satisfied due to Conditions \eqref{momgenZ} and \eqref{expectations} of Theorem \ref{LargeDev}: For each $n\in \N$ and each realisation of $(W_j)_{j\in\N}$ $\Psi_n^{\WW}(\vt)$ is a convex function. Furthermore, 
\begin{equation}\label{max1}
\Psi_n^{\WW}(\vt)\leq \max\{\Psi_n^{\WW}(\vt_*),\Psi_n^{\WW}(-\vt_*)\}  
\end{equation}
and 
\begin{equation}\label{max2}
 \lim_{n\to\infty} \max\{\Psi_n^{\WW}(\vt_*),\Psi_n^{\WW}(-\vt_*)\} =
  \max\{\E \left[f(W_1\vt_*)\right],\E\left[f(-W_1\vt_*)\right)]\}, \quad\hbox{\rm a.s.}
\end{equation}
This implies that  Assumption \ref{ass1} is satisfied. 
To prove that  Assumption \ref{ass2} holds,  note that, by the law of large numbers,
\begin{equation}\Eq(plup.1)
\lim_{n\to\infty} \frac{d}{d\vt}\Psi_n^{\WW}(0)= 
\lim_{n\to\infty} \frac 1n\E^{\WW}[S_n]= \E[W_1]\E[Z_1], \quad\hbox{\rm a.s.} 
\end{equation}
Next, by convexity, and again the law of large numbers
\begin{equation}\Eq(plup.2)
\liminf_{n\to\infty}\sup_{\vt\in [0,\vt_{*}]}\frac{d}{d\vt}\Psi_n^{\WW}(\vt) 
= \liminf_{n\to\infty} \frac{d}{d\vt}\Psi_n^{\WW}(\vt_*) = \E[W_1f'(\vt_*W_1)] , \quad\hbox{\rm a.s.} 
\end{equation}
Recall that $\vt_n^{\WW}(a)$ is defined as the solution of 
\begin{equation}\label{Thetadet}
a=\frac 1n\sum_{j=1}^{n}\frac d{d\vt}\log M(W_j \vt) = \frac 1n \sum_{j=1}^n W_j f'(\vt W_j) . 
\end{equation}
For $n$ large enough, the solution $\vt_n^{\WW}(a)$ exists for  $a\in J$ and is unique since  the logarithmic moment generating function $\Psi_n^{\WW}$ is strictly convex. Again by monotonicity of $\frac{d}{d\vt}\Psi_n^{\WW}(\vt)$ in $\vt$, and 
because of \eqv(plup.1) and \eqv(plup.2),  for $a\in J$,  $\vt_n^\WW(a)\in (0,\vt_*)$, 
almost surely, for $n$ large enough.
Thus Assumption \ref{ass2} is satisfied.

In order to establish Condition \eqref{1} of Theorem \ref{abwthm} we prove the following
\begin{lemma} \TH(plop.6)
$\P\big(\forall a\in J: \lim_{n \to\infty}\vt_n^{\WW}(a)=\vt(a)\big)=1$.
\end{lemma}
\begin{proof} 
First, using that  $g'(\vt)$ is continuous and monotone increasing
\begin{eqnarray}\nonumber
	&&\P\left(\forall a \in J: \lim_{n\to\infty}|\vt_n^{\WW}(a)-\vt(a)|=0\right) \\\nonumber
	&&= \P\Bigg(\forall a \in J: \lim_{n\to\infty}\Big|g'(\vt_n^{\WW}(a))-\frac1n\sum_{j=1}^nW_jf'(\vt_n^{\WW}(a)W_j) \\\nonumber
	&&\qquad\qquad-g'(\vt(a))+ \frac1n\sum_{j=1}^nW_jf'(\vt_n^{\WW}(a)W_j)\Big|=0\Bigg)\\
&&=	 \P\left(\forall a \in J: \lim_{n\to\infty}\Big|g'(\vt_n^{\WW}(a))-\frac1n\sum_{j=1}^nW_jf'(\vt_n^{\WW}(a)W_j) \Big|=0\right),\label{uni1}
\end{eqnarray}	
where we used that,
by definition of $\vt(a)$ and $\vt_n(a)$, 
\begin{equation} 
	\frac1n\sum_{j=1}^nW_jf'(\vt_n^{\WW}(a)W_j)=g'(\vt(a))=a.
\end{equation}
Since we have seen that for $a\in J$, $\vt(a)\in [0, \vt_*]$ and, for  $n$ large enough,
 $\vt_n^\WW(a)\in [0,\vt_*]$, 
the last line in   \eqref{uni1} is bounded from below by
\be\Eq(plop.3)
\P\left(\sup_{\vt\in[0,\vt_*]}\lim_{n\to\infty}\left|\frac1n \sum_{j=1}^nW_jf'(\vt W_j)- \E[W_1f'(\vt W_1)]\right|=0\right).
\ee
The following facts are true:
\begin{enumerate}
\item \label{int} By Condition \eqref{expectations} of Theorem \ref{LargeDev} $W_1f'(\vt W_1)$ is integrable, for each $\vt\in [0,\vt_*]$. 
\item $W_1(\o)f'(\vt W_1(\o))$ is a continuous function of $\vt$, $\forall \o\in \O$.  \label{continuity}
\item $W_1f'(\vt W_1)$ is monotone increasing in $\vt$ since $\frac{d}{d\vt}(W_1f'(\vt W_1))>0$. \label{increasing}
\item \label{plop.4}  \eqref{int}, \eqref{continuity}, and \eqref{increasing} imply, by dominated convergence, that, for all $\vt\in [0,\vt_*]$, 
\begin{equation}
\lim_{\d\downarrow 0}\E\left[\sup_{\bar{\vt}\in B_{\d}(\vt)}W_1f'(\bar{\vt} W_1)- \inf_{\bar{\vt}\in B_{\d}(\vt)}W_1f'(\bar{\vt} W_1)\right] = 0. 
\end{equation}
\end{enumerate}
Note that \ref{plop.4} implies that, for all $\vt\in [0,\vt_*]$ and for all  $\e>0$,  
there exists a $\d=\d(\e,\vt)$, such that 
\be\label{conv}
\E\left[\sup_{\bar{\vt}\in B_{\d(\e,\vt)}(\vt)}W_1f'(\bar{\vt} W_1)- \inf_{\bar{\vt}\in B_{\d(\e,\vt)}(\vt)}W_1f'(\bar{\vt} W_1)\right] <\e.
\ee
The collection  $\{B_{\d(\e,\vt)}(\vt)\}_{\vt\in[0,\vt_*]}$ is an open cover of $[0,\vt_*]$,
 and since $[0,\vt_*]$ is compact we can choose a finite subcover, $\{B_{\d(\e,\vt_k)}(\vt_k)\}_{1\leq k \leq K}$. Therefore
\begin{eqnarray}\label{uni2}
	&&\sup_{\vt\in[0,\vt_*]}\left\{\left|\frac1n \sum_{j=1}^nW_jf'(\vt W_j)- \E[W_1f'(\vt W_1)]\right|\right\} \\\nonumber
	&&= \max_{1\leq k\leq K}\sup_{\vt\in B_{\d(\e,\vt_k)}(\vt_k)}\left\{\left|\frac1n \sum_{j=1}^nW_jf'(\vt W_j)- \E[W_1f'(\vt W_1)]\right|\right\} \\\nonumber
	&&\leq  \max_{1\leq k\leq K}\left\{\left|\frac1n \sum_{j=1}^n \sup_{\vt\in B_{\d(\e,\vt_k)}(\vt_k)} W_jf'(\vt W_j)-  \inf_{\vt\in B_{\d(\e,\vt_k)}(\vt_k)} \E[W_1f'(\vt W_1)]\right|\right\}.
\end{eqnarray}
Since $\sup_{\vt\in B_{\d(\e,\vt_k)}(\vt_k)} W_jf'(\vt W_j)\leq W_jf'(\vt_*W_j)$ is integrable, the strong law of large numbers applies and \eqref{uni2} converges almost surely to
\begin{equation}
	 \max_{1\leq k\leq K}\left\{\left|\E\left[\sup_{\vt\in B_{\d(\e,\vt_k)}(\vt_k)} W_1f'(\vt W_1)\right] -  \inf_{\vt\in B_{\d(\e,\vt_k)}(\vt_k)} \E[W_1f'(\vt W_1)]\right|\right\},
\end{equation}
which in turn, due to  \eqref{conv},  is bounded from above by
\begin{equation}
	\max_{1\leq k\leq K}\left\{\left| \E\left[\sup_{\vt\in B_{\d(\e,\vt_k)}(\vt_k)} W_1f'(\vt W_1) -  \inf_{\vt\in B_{\d(\e,\vt_k)}(\vt_k)} W_1f'(\vt W_1)\right]\right|\right\} < \e.
\end{equation} 
Thus, $\vt_n^{\WW}(a)$ converges almost surely to $\vt(a)$. 
\end{proof}
But for  $a\in J$, we know that  $\vt(a)>0$, and since $\vt^\WW_n(a)$ converges to $
\vt(a)$, a.s., a fortiori, Condition \eqref{1} of Theorem \ref{abwthm} is  satisfied, a.s.

Next we show that  Condition \eqref{2} of Theorem \ref{abwthm}  is also satisfied, 
almost surely. 
 To see this, write  
\begin{eqnarray}\nonumber
	&&\left.\left( \frac{d^2}{d \vt ^2}\Psi_n^{\WW}(\vt)\right) \right|_{\vt=\vt_n^{\WW}(a)}\\
	\nonumber
	&&=\frac{1}{n}\sum_{j=1}^n\tfrac{\E[W_j^2Z_j^2e^{\vt_n^{\WW}(a)W_jZ_j}|\WW]
	\E[e^{\vt_n^{\WW}(a)W_jZ_j}|\WW]-\left(\E[W_jZ_je^{\vt_n^{\WW}(a)W_jZ_j}|\WW]
\right)^2}{(\E[e^{\vt_n^{\WW}(a)W_jZ_j}|\WW])^2}\\
	&&= \frac 1n\sum_{j=1}^n\Va ^{\vt_n^{\WW}(a)}[W_jZ_1|\WW].
\end{eqnarray}
The conditional variance $\Va ^{\vt_n^{\WW}(a)}[W_jZ_j|\WW]$ is clearly positive with positive probability, since we assumed the distribution of $Z_1$  to be non-degenerate and 
$W_j$ is non-zero with positive probability. We need to 
show that also the  infimum over $n\in \N$ is  strictly positive. Note that 
\begin{equation}
	\Psi''(\vt(a))=\E[\Va^{\vt(a)}[W_1Z_1|\WW]]>0.
\end{equation}
We need the following lemma.
\begin{lemma}\TH(plop.5)
	$\P\left(\forall a\in J :\lim_{n \to \infty }\Psi_n''(\vt_n^{\WW}(a))=\Psi''(\vt(a))\right)=1$.
\end{lemma}
\begin{proof}
Since trivially
\begin{eqnarray}
	&&|\Psi_n''(\vt_n^{\WW}(a))-\Psi''(\vt(a))| \nonumber \\
	&&\leq |\Psi_n''(\vt_n^{\WW}(a))-\Psi''(\vt_n^{\WW}(a))| + |\Psi''(\vt_n^{\WW}(a))-\Psi''(\vt(a))|,
\end{eqnarray}
Lemma \thv(plop.5) follows if both
\begin{equation}\label{term1}
	\P\left(\forall a \in  J: \lim_{n\to\infty}|\Psi_n''(\vt_n^{\WW}(a))-\Psi''(\vt_n^{\WW}(a))|=0\right)=1
\end{equation}
and
\begin{equation}\label{term2}
	\P\left(\forall a \in J : \lim_{n\to\infty}|\Psi''(\vt_n^{\WW}(a))-\Psi''(\vt(a))|=0\right)=1.
\end{equation}
Now, $\Psi''(\vt)$ is a  continuous function of $\vt$ and uniformly continuous on the compact interval $[0,\vt_*]$. This implies  that 
\begin{equation}
	\forall \e>0 \, \exists \d=\d(\e) : \forall \vt,\vt':|\vt-\vt'|<\d \quad |\Psi''(\vt)-\Psi''(\vt')|<\e.
\end{equation}
From  the uniform almost sure convergence of $\vt_n^{\WW}(a)$ to $\vt(a)$,  
it follows that
\begin{equation}
	\forall \d>0 \, \exists N=N(\o,\d) : | \vt_n^{\WW}(a)-\vt(a)|<\d, 
\end{equation}
which in turn  implies that 
\begin{equation}
	\forall n\geq N: |\Psi''(\vt_n^{\WW}(a))-\Psi''(\vt(a))|<\e.
\end{equation}
Therefore, Equation \eqref{term2} holds. The proof of \eqref{term1} is very similar 
to that of Lemma \thv(plop.6). The difference is that we cannot use monotonicity to 
obtain a majorant for $W_1^2f''(\vt W_1)$, but instead use 
 Condition \eqref{majorant} of Theorem \ref{LargeDev}. Again, as in \eqv(plop.3), 
\begin{eqnarray}\nonumber
	&&\P\left(\forall a \in J: \lim_{n\to\infty}|\Psi_n''(\vt_n^{\WW}(a))-\Psi''(\vt_n^{\WW}(a))|=0\right)\\
	&&\geq \P\left(\sup_{\vt\in[0,\vt_*]} \lim_{n\to\infty}|\Psi_n''(\vt)-\Psi''(\vt)|=0\right).
\end{eqnarray}
Moreover, the following facts are true: 
\begin{enumerate}
\item By Condition \eqref{majorant} of Theorem \ref{LargeDev}, $W_1^2f''(\vt W_1)\leq F(W_1)$ and $\E[F(W_1)]<\infty$. \label{int2}
\item $W_1^2(\o)f''(\vt W_1(\o))$ is a continuous function of $\vt$ $\forall \o\in \O$.  \label{continuity2}
\item \label{plop.7} From \eqref{int2} and \eqref{continuity2} it follows by dominated convergence that 
for all $\vt\in [0,\vt_*]$ that 
\begin{equation}
	\lim_{\d\downarrow 0} \E\left[\sup_{\bar{\vt}\in B_{\d}(\vt)}W_1^2f''(\bar{\vt} W_1)- \inf_{\bar{\vt}\in B_{\d}(\vt)}W_1^2f''(\bar{\vt} W_1)\right] = 0.
\end{equation}
\end{enumerate}
The proof of Lemma \thv(plop.5) proceeds from here exactly as the proof of Lemma
\thv(plop.6), just replacing $f'$ by $f''$ and $W_1$ by $W_1^2$.
\end{proof}

 Condition (\ref{2}) of Theorem \ref{abwthm} now follows immediately.

 Next we  show that  Condition (\ref{3}) is satisfied. 
We want to show that $\forall 0<\d_1<\d_2<\infty $
\begin{equation}\label{asconvcf}
	\P\left( \forall a\in J:
	 \lim_{n\to \infty}\sqrt n \sup_{\d_1\leq|t|\leq\d_2\vt_n^{\WW}(a)}\Bigg|\frac{\Phi_n^{\WW}(\vt_n^{\WW}(a)+it)}{\Phi_n^{\WW}(\vt_n^{\WW}(a))}\Bigg|=0\right)=1.
\end{equation}
As above we bound the probability in \eqv(asconvcf) from below by
\begin{equation}\label{asconvcf2}
\P\left(\lim_{n\to \infty}\sqrt n \sup_{\vt\in[0,\vt_*]}\sup_{\d_1\leq|t|\leq\d_2\vt}\left|\frac{\Phi_n^{\WW}(\vt+it)}{\Phi_n^{\WW}(\vt)}\right|=0\right).
\end{equation}
Therefore, \eqv(asconvcf) follows from the first  Borel-Cantelli lemma if,
 for each $\d>0$,
\begin{equation}\label{sum}
	\sum_{n=1}^{\infty} \P\left(\sqrt n  \sup_{\vt\in[0,\vt_*]}\sup_{\d_1\leq|t|\leq\d_2\vt}\left|\frac{\Phi_n^{\WW}(\vt+it)}{\Phi_n^{\WW}(\vt)}\right| >\d\right)<\infty .
\end{equation}
Note that
\begin{equation}\label{product}
	\left|\frac{\Phi_n^{\WW}(\vt+it)}{\Phi_n^{\WW}(\vt)}\right|=\prod_{j=1}^n\left|\frac  {M(W_j(\vt+it))}{M(W_j\vt)}\right|
\end{equation}
is a product of functions with absolute value less or equal to $1$. Each factor is the characteristic function of a tilted $Z_j$. According to a result of Feller \cite{Feller2} there are 3 classes of characteristic functions.
\begin{lemma}[Lemma 4 in  Chapter XV in \cite{Feller2}] Let $\phi$ be the characteristic  function of a probability distribution function $F$. 
Then one of the following must hold:
\begin{enumerate}
\item \label{plop.10}$|\phi(\z)|<1$ for all $\z\neq 0$.
\item \label{plop.11} $|\phi(\l)|=1$ and $|\phi(\z)|<1$ for $0<\z<\l$. In this case $\phi$ has period $\l$ and there exists a real number $b$ such that $F(x+b)$ is arithmetic with span $h=\nicefrac{2\pi}{\l}$.
\item \label{plop.12} $|\phi(\z)|=1$ for all $\z$. In this case $\phi(\z)=\rm{e}^{ib\z}$ and $F$ is concentrated at the point $b$.
\end{enumerate}
\end{lemma}
Case \eqv(plop.12) is excluded by assumption. Under Condition \eqref{Nolattice} of Theorem 
\ref{LargeDev} we are in Case \eqv(plop.10). In this case 
 it is rather easy to verify Equation \eqref{asconvcf}. Namely, observe that there exists
 $0<\rho<1$, such that for all $\vt\in [0,\vt_*]$, 
 for all $\d_1\leq t\leq \d_2\vt_*$, whenever $K^{-1}\leq |W_j|\leq K$, 
 for some $0<K<\infty$,
 \be\Eq(plop.13)
 \left|\frac  {M(W_j(\vt+it))}{M(W_j\vt)}\right|<1-\rho.
 \ee
 This implies that, for $\vt$ as specified, 
  \be\Eq(plop.13.1)
 \left|\frac  {M(W_j(\vt+it))}{M(W_j\vt)}\right |\leq \left(1-\rho\right)^{\1_{ \{\frac{1}{K}\leq|W_j|\leq K \}}}.
 \ee

Therefore,
\begin{eqnarray}\nonumber
&& \P\left(\sqrt n  \sup_{\vt\in[0,\vt_*]}\sup_{\d_1\leq|t|\leq\d_2\vt}\prod_{j=1}^n\left|\frac  {M(W_j(\vt+it))}{M(W_j\vt)}\right|>\d\right) \\
&&\leq \P\left(\sqrt n(1-\r)^{\sum_{j=1}^n \mathds{1}_{ \{\frac{1}{K}\leq|W_j|\leq K \}}}>\d\right), \label{uni10}
\end{eqnarray}
where $K$ is chosen such that $\P\left(\frac{1}{K}\leq|W_j|\leq K \right)>0$. 
With $c_n\equiv\frac{\log \d-\frac 12\log n}{\log(1-\r)}$, the probability in the second line of  \eqref{uni10} is equal to 
\begin{eqnarray}\nonumber
 &&\P\left(\sum_{j=1}^n \mathds{1}_{\{\frac{1}{K}\leq|W_j|\leq K\}} < c_n\right)\\
 \nonumber
&&\leq \sum_{k=1}^{\left\lceil c_n\right\rceil} \binom{n}{k}\P\left(\frac{1}{K}\leq|W_j|\leq K \right)^k\left[1-\P\left(\frac{1}{K}\leq|W_j|\leq K \right)\right]^{n-k}\\
&&\leq \left\lceil c_n\right\rceil\binom {n}{\left\lceil c_n\right\rceil}\left[1-\P\left(\frac{1}{K}\leq|W_j|\leq K \right)\right]^{n-\lfloor c_n\rfloor}.
\end{eqnarray}
 Since $\binom {n}{\left\lceil c_n\right\rceil}\sim n^{C\ln n}$ for a constant $C$,  this is summable in $n$ and \eqref{sum} holds.

Case \eqv(plop.11) of  lattice-valued random variables $Z_j$, which corresponds to Condition \eqref{Nolattice2} of Theorem \ref{LargeDev}, is more subtle.  Each 
of the factors in the product in \eqv(product) is a periodic function, which is equal to $1$ if and only if $W_j t \in \{k\l, k\in \Z\}$, where $\l$ is the period of this function.
This implies that each factor is smaller than $1$ if \mbox{$W_j\notin \{\nicefrac{k\l}{t},
 k\in \Z\}$}. The points of this set do not depend on $\vt$ and have the 
smallest distance to each other if $t$ is maximal, i.e. $t=\d_2\vt_*$. 
Each factor is \emph{strictly} smaller than $1$ if $tW_j$ does not lie in a finite interval around one of these points. 
We choose these intervals as follows. 
Let  $\wt{\d}=\frac{1}{8\d_2\vt_*}$ and define the intervals
\begin{equation}
	I(k,t,\wt{\d})\equiv \left[\frac{k\l}{t}-\wt{\d},\frac{k\l}{t}+\wt{\d}\right].
\end{equation} 
This intervals are disjoint and consecutive intervals are separated by a distance at least 
 $6\wt{\d}$ from  each other.  Then, for all  $\vt\in[0,\vt_*]$ there exists $0<\rho(\vt)<1$ ,  independent of 
 $t$, such that 
 \begin{equation}\label{uni4}
	\left|\frac  {M(W_j(\vt+it))}{M(W_j\vt)}\right|\leq \left(1-\r(\vt)\right)^{\mathds{1}_{\{|W_j|\notin \cup_{k\in \Z} I(k,t,\wt{\d})\}}}.
\end{equation}
Furthermore, $\left|\frac{M(\th+it)}{M(\th)}\right|$ is continuous in $\th$, and thus its supremum over compact intervals is attained. Thus,  for any $C>0$ there  exists
  $\bar{\r}=\bar{\r}(C)>0$ such that, for all $\vt\in [0,\vt_*]$,
\begin{equation}\label{rho}
	\left|\frac  {M(W_j(\vt+it))}{M(W_j\vt)}\right|\leq \left(1-\bar{\r}\right)^{\mathds{1}_{\{W_j\in[-C,C] \backslash \cup_{k\in \Z} I(k,t,\wt{\d})\}}}.
\end{equation}
We choose $C$  such that  the interval $[c,d]$ from Hypothesis  \eqv(Nolattice) is 
contained in $ [-C,C]$.
Then we get with Equation \eqref{uni4} and \eqv(rho) that 
\begin{eqnarray}\label{inf}\nonumber
	&& \P\Bigg(\sqrt n \sup_{\vt\in [0,\vt_*]}\sup_{\d_1\leq|t|\leq\d_2\vt }\prod_{j=1}^n
	\Bigg|\frac{M(W_j(\vt+it))}{M(W_j\vt)}\Bigg|>\d\Bigg) \\\nonumber 
	&&\leq \P\Bigg(\sqrt n \sup_{\vt\in [0,\vt_*]}\sup_{\d_1\leq|t|\leq\d_2\vt_* }\prod_{j=1}^n \left(1-\bar{\r}\right)^{\mathds{1}_{\{W_j\in[-C,C] \backslash \cup_{k\in \Z} I(k,t,\wt{\d})\}}}>\d\Bigg) \nonumber \\
	&& = \P\left(\sqrt n \left(1-\bar{\r}\right)^{\inf_{\d_1\leq|t|\leq\d_2\vt_*}\sum_{j=1}^n\mathds{1}_{\{W_j\in[-C,C] \backslash \cup_{k\in \Z} I(k,t,\wt{\d})\}}}>\d\right). 
\end{eqnarray}
With $c_n\equiv \tfrac{\log \d-\frac 12\log n}{\log(1-\bar{\r})}$ Equation  \eqref{inf} can be rewritten as
\begin{eqnarray}\label{uni6}
	&& \P\Big(\inf_{\d_1\leq|t|\leq\d_2\vt_*}\sum_{j=1}^n\mathds{1}_{\{W_j\in[-C,C] \backslash \cup_{k\in \Z} I(k,t,\wt{\d})\}}<c_n\Big)  \nonumber \\
	&&\leq \P\Big(\inf_{\d_1\leq|t|\leq\d_2\vt_*}\sum_{j=1}^n\mathds{1}_{\{W_j\in ([-C,C] \cap [c,d] )\backslash \cup_{k\in \Z} I(k,t,\wt{\d})\}}<c_n\Big).
\end{eqnarray}
\eqref{uni6} is summable over  $n$ since the number of $W_j$ contained in the ``good'' sets is of order $n$, i.e. $\# \{j:W_j\in [c,d]\backslash \cup_{k\in \Z} I(k,t,\wt{\d})\}=\OO(n)$.
Define 
\be\Eq(plop.20)
K(t)=\#\{k: I(k,t,\wt{\d})\cap [c,d]\neq \emptyset\},
\ee
and let $k_1,\dots k_{K(t)}$ enumerate the intervals contained in $[c,d]$. 
Let $m_1(t),\dots, m_{K(t)}(t)$ be chosen such that $W_{m_i(t)}\in I(k_i,t,\wt{\d})$. Note 
that $m_i(t)$ are random. The probability in the last line of   \eqref{uni6} is bounded from 
above by
\begin{equation}\label{translate}
  \P\left(\inf_{\d_1\leq|t|\leq\d_2\vt_*} \sum_{j=1}^n\mathds{1}_{\{W_j\in [c,d],|W_j-
  W_{m_1(t)}|>2\wt{\d}, \dots, |W_j-W_{m_{K(t)}(t)}|>2\wt{\d}\}}\leq c_n\right).
\end{equation}
Since there are only finitely many intervals of length $2\wt{\d}$ with distance $6\wt{\d}$
 to each other in $[c,d]$, there exists $K<\infty$ such that $\sup_{t\in [\d_1,\d_2\vt_*]}
 K(t)<K$.
Thus, the probability in  \eqref{translate} is not larger than
\begin{eqnarray}
 && \P\Bigg( \exists_{m_1,\dots, m_{K}\in\{1,\dots, n\}} : \nonumber
 \sum_{j=1}^n\mathds{1}_{\left\{W_j\in [c,d],|W_j-W_{m_1}|>2\wt{\d}, \dots, |W_j-W_{m_{K}}|>2\wt{\d}\right\}}\leq c_n\Bigg) \nonumber \\
&&\leq \sum_{m_1,\dots, m_K=1}^n\P\Bigg( \sum_{j=1}^n\mathds{1}_{\left\{W_j\in [c,d],|W_j-W_{m_1}|>2\wt{\d}, \dots, |W_j-W_{m_{K}}|>2\wt{\d}\right\}}\leq c_n\Bigg) \nonumber \\
&&\leq n^K \P\Bigg( \sum_{j=1}^n\mathds{1}_{\left\{W_j\in [c,d],|W_j-W_{m_1}|>2\wt{\d}, \dots, |W_j-W_{m_{K}}|>2\wt{\d}\right\}}\leq c_n\Bigg).\label{boundsum2}
\end{eqnarray}
The indicator function vanishes whenever $j=m_i$ with $i\in \{1,\dots, K\}$. Thus, \begin{eqnarray}\label{uni7}
&&\P\Bigg( \sum_{j=1}^n\mathds{1}_{\left\{W_j\in [c,d],|W_j-W_{m_1}|>2\wt{\d}, \dots, |W_j-W_{m_{K}}|>2\wt{\d}\right\}}\leq c_n\Bigg)  \nonumber\\
&&= \P\Bigg( \sum_{j \not\in \{ m_1, \dots, m_K\}}^n\mathds{1}_{\left\{W_j\in [c,d],|W_j-W_{m_1}|>2\wt{\d}, \dots, |W_j-W_{m_{K}}|>2\wt{\d}\right\}}\leq c_n\Bigg)\nonumber  \\
&&=\P\Bigg( \sum_{j=K}^{n}\mathds{1}_{\left\{W_j\in [c,d],|W_j-W_{1}|>2\wt{\d}, \dots, |W_j-W_{K}|>2\wt{\d}\right\}}\leq c_n\Bigg)\label{boundsum}
\end{eqnarray}
due to the i.i.d. assumption.
\eqref{uni7} is equal to 
\begin{equation}\label{uni8}
\sum_{l=0}^{\left\lceil c_n\right\rceil} \binom{n-K}{l}\P(A)^l\left(1-\P(A)\right)^{n-K-l}\leq \left\lceil c_n\right\rceil \binom{n-K}{\left\lceil c_n\right\rceil} (1-\P(A))^{n-K-\left\lceil c_n\right\rceil}. 
\end{equation}
Here $A$ is the event
\begin{equation}
 A=\left\{W\in [c,d],|W-W_{1}|>2\wt{\d}, \dots, |W-W_{K}|>2\wt{\d}\right\},
 \end{equation}
 where $W$ is an independent copy of $W_1$. 
We show that  $\P(A)$ is strictly positive. 
\begin{eqnarray}
\P(A)
&=& \int_{[c,d]} \P\left(|W-W_{1}|>2\wt{\d}, \dots, |W-W_{K}|>2\wt{\d}\, \Big|\, W\right)dP_{W}\\\nonumber
&\geq& \int_{[c,d]} \P\left(W_i \in [W-2\wt{\d},W+2\wt{\d}]^c\cap [c,d] , \forall  i \in \{1,\dots, K\}\right)dP_{W},
\end{eqnarray}
where $P_{W}$ denotes the distribution of $W$.
Since the random variables  $W_1,\dots, W_K, W$ are independent of each other,
 this is equal to
\be  \int_{[c,d]}\P\left(W_1 \in [W-2\wt{\d},W+2\wt{\d}]^c\cap [c,d]|W\right)^KdP_{W},
\ee
and due to the lower bound on the density of  $P_W$ postulated in Hypothesis
\eqv(Nolattice), this in turn is bounded from below by
\be
(p(d-c-4\wt{\d}))^K\int_{[c,d]} dP_{W}
\geq (d-c)p^{K+1}(d-c-4\wt{\d})^K \equiv \wt{p} \in (0,1].\label{uni9}
\ee
Combining Equations \eqref{uni8} and \eqref{uni9} we obtain
\be
	\eqref{boundsum2} \leq n^K\left\lceil c_n\right\rceil\binom{n}{\left\lceil c_n
	\right\rceil}\wt{p}^{n-K-\left\lceil c_n\right\rceil}
\ee
which is summable over $n$, as desired. Thus all hypotheses of Theorem 
\thv(abwthm) are 
satisfied with probability one, uniformly in $a\in J$, and so 
the conclusion of  Theorem \thv(LargeDev) follows. 
\end{proof}


\section{Proof of Theorem \ref{FCLT}}\label{proofs2}

In order to prove Theorem \ref{FCLT} we need  the joint weak 
convergence of the process $X_n$, defined in \eqv(plop.30) and its derivatives, 
as stated in Lemma \thv(JWC).
 Define on the closure, $\bar J$ of the interval $J$ (recall the definition of $J$ in 
 Theorem \thv(LargeDev)), the processes $(\wh X^n_a)_{a\in \bar J}$, $n\in\N$, via
\begin{equation}\label{process}
	\wh X_a^n\equiv (X_n(\vt(a)),X_n'(\vt(a)),X_n''(\vt(a))).
\end{equation}
\begin{lemma} \label{JWC}
The family of processes
$(\wh X_a^n)_{a\in \bar J}$  defined on 
$\left(C(\bar J,\R^3), \BB( C(\bar J,\R^3)\right)$, converges weakly, as $n\to\infty$, to a 
process $(\wh X_a)_{a\in \bar J}$ on the same space,
  if there exists $c>0$, such that, for all $a\in \bar J$,  $g''(\vt(a))>c$,  and  if Assumption \ref{Cov} is satisfied.
\end{lemma}
\begin{proof}  As usual, we prove convergence of  
 the finite dimensional distributions and tightness.
 
 More precisely,  we have to check that:
\begin{enumerate}
\item \label{convfdd} $(\wh X_a^n)_{a\in \bar J}$ converges in finite dimensional distribution. 
\item \label{starttight}The family of initial distributions, i.e. the distributions of 
$\wh X_b^n$, where  $b\equiv \E[Z_1W_1]$, 
is tight. 
\item \label{tightness} There exists $C>0$ independent of $a$ and $n$ such that 
		\begin{equation}
		\E\left[\|\wh X_{a+h}^n-\wh X_a^n\|^2\right]\leq C|h|^2,
		\end{equation} 
		which is a Kolmogorov-Chentsov criterion for tightness, see 
		 \cite[Corollary 14.9]{KAL}.
\end{enumerate}

First, we consider the finite dimensional distributions. Let 
\begin{eqnarray}
Y_{a,j} &\equiv& f(\vt(a)W_j)- \E\left[f(\vt(a)W_j)\right] \nonumber \\
 Y_{a,j}' &\equiv &W_jf'(\vt(a)W_j)-\E[W_jf'(\vt(a)W_j)] \quad \text{ and } \nonumber \\
 Y_{a,j}'' &\equiv & W_j^2f''(\vt(a)W_j)- \E[W_j^2f''(\vt(a)W_j)]. \label{bla.1}
 \end{eqnarray}
Moreover, let $\ell\in \N$, 
$a_1<a_2<\dots<a_\ell \in \bar J$ and
\begin{equation}\label{bla.2}
 \chi_j\equiv \Big(Y_{a_1,j},Y_{a_1,j}',Y_{a_1,j}'',\dots,Y_{a_\ell,j},Y_{a_\ell,j}',Y_{a_\ell,j}''\Big)\in \R^{3\ell}.
 \end{equation}
 These vectors are independent for different $j$ and  its components 
 $(\chi_j)_k, 1\leq k\leq 3\ell$, have covariances  \mbox{$\Cov((\chi_j)_k,
 (\chi_j)_m)=\CC_{km}<C$} for all $k,m \in \{1,\dots, 3\ell\}$,  according to Assumption 
 \ref{Cov}. Therefore, $\sfrac 1{\sqrt n}\sum_{j=1}^n \chi_j$ converges, as $n\to \infty$, 
 to  the $3\ell$-dimensional Gaussian vector with mean zero and covariance matrix
 $\CC$ by the  central limit theorem. This proves   convergence 
 of the  finite dimensional distributions of $(\wh X^n_a)_{a\in \bar J}$. 
 
The family of initial distributions is given by the random variables evaluated in $\vt(b)$. This family is seen to be  tight using Chebychev's inequality
\begin{eqnarray} \label{start}
\P\left(\left\| X_{b}^n\right\|_2>C\right) &\leq &\frac{\mathbb V\left[\sqrt{X_n(\vt(b))^2+(X_n'(\vt(b)))^2+(X_n''(\vt(b)))^2}\right]}{C^2} \nonumber \\
& \leq &\frac{\E\left[X_n(\vt(b))^2+(X_n'(\vt(b)))^2+(X_n''(\vt(b)))^2\right]}{C^2}.
\end{eqnarray}
which is finite by Assumption \thv(Cov).
For each $\e$ we can choose $C$ large enough such that \eqref{start} $<\e$. 
It remains to check Condition \eqref{tightness}. Since
\begin{eqnarray}\label{tightX} \nonumber 
\E\left[\|\wh X_{a+h}^n-\wh X_a^n\|^2\right]&=&\E\left[(X_n(\vt(a+h))-X_n(\vt(a)))^2\right]\\\nonumber
&&+\E\left[[X_n'(\vt(a+h))-X_n'(\vt(a)))^2\right] \nonumber \\
	&&
	 +\E\left[(X_n''(\vt(a+h))-X_n''(\vt(a)))^2\right],
\end{eqnarray}
we need to show that each of the three terms on the right-hand side is of order $h^2$.
Note that \mbox{$\E\left[[X_n(\vt(a+h))-X_n(\vt(a))]^2\right]\leq C|h|^2$} if 
\begin{equation}\label{versiontightness}
	\E\left[\left( \tfrac {d}{da}X_n(\vt(a))\right)^2\right]\leq C.
\end{equation}
Since $X_n(\vt(a))=\frac{1}{\sqrt n}\sum_{j=1}^{n}Y_{a,j}$, 
\begin{equation}\label{bla.3}
	\E\left[\left( \tfrac {d}{da}X_n(\vt(a))\right)^2\right]=\frac{1}{n}\sum_{j=1}^n\E\left[\left(\tfrac{d}{da}Y_{a,j}\right)^2\right].
\end{equation}
Each summand can be controlled by
\begin{eqnarray}
&&\E\left[\left(\tfrac{d}{da}Y_{a,j}\right)^2\right] = \E\left[\left(\tfrac{d}{da} f(\vt(a)W_j)\right)^2\right]-\left(\E\left[\tfrac{d}{da}f(\vt(a)W_j)\right]\right)^2\nonumber  \\
&&= \left(\tfrac{d}{da}\vt(a)\right)^2\left(\E\left[W_j^2f'(\vt(a)W_j)^2\right]-\left(\E\left[W_jf'(\vt(a)W_j)\right]\right)^2\right) \nonumber \\
&&=  \left(\tfrac{d}{da}\vt(a)\right)^2\Va\left[W_jf'(\vt(a)W_j)\right].\label{bla.4}
\end{eqnarray}
By the implicit function theorem, 
\begin{equation}\label{bla.5}
\frac {d}{da}\vt(a)=\left(g''(\vt(a))\right)^{-1}.
\end{equation}
Thus,  Equation \eqref{versiontightness} holds since $g''(\vt(a))>c$ by 
assumption  and $\Va\left[W_jf'(\vt(a)W_j)\right]$ is bounded by Assumption \ref{Cov}. The bounds for the remaining terms follow in the same way by controlling   the derivatives of  $X_n'(\vt(a))$ and $X_n''(\vt(a))$. We obtain
\begin{eqnarray}\label{bla.6}
\E\left[\left(\tfrac{d}{da}Y_{a,j}'\right)^2\right] &=& \E\left[\left(\tfrac{d}{da} W_jf'(\vt(a)W_j)\right)\right]-\left(\E\left[\tfrac{d}{da}W_jf'(\vt(a)W_j)\right]\right)^2 \nonumber \\
&=&\left(\tfrac{d}{da}\vt(a)\right)^2\Va\left[W_j^2f''(\vt(a)W_j)\right]
\end{eqnarray}
and
\begin{eqnarray}\label{bla.7}
\E\left[\left(\tfrac{d}{da}Y_{a,j}''\right)^2\right] &=& \E\left[\left(\tfrac{d}{da} W_j^2f''(\vt(a)W_j)\right)\right]-\left(\E\left[\tfrac{d}{da}W_j^2f''(\vt(a)W_j)\right]\right)^2 \nonumber\\
&=&\left(\tfrac{d}{da}\vt(a)\right)^2\Va\left[W_j^3f'''(\vt(a)W_j)\right].
\end{eqnarray}
In both formulae the right hand sides are bounded due to Assumption \thv(Cov). 
This proves the lemma.
\end{proof}
\begin{proof}[Proof of Theorem \ref{FCLT}]
Recall that   $\vt_n^{\WW}(a)$  is determined as the solution of the  equation 
\begin{equation}\label{bla.8}
a =g'(\vt) + \frac{1}{\sqrt n} X_n'(\vt).
\end{equation}
Write $\vt_n^{\WW}(a)\equiv \vt(a)+\d^n(a)$, where $\vt(a)$  is defined as the solution of 
\begin{equation}\label{bla.9}
a = g'(\vt).
\end{equation}
Note that $\vt(a)$ is deterministic while $\d^n(a)$ is random and $\WW$-measurable . 
The rate function can be rewritten as 
\begin{equation}\label{bla.10}
I_n^{\WW}(a) = a(\vt(a)+\d^n(a))-g( \vt(a)+\d^n(a)) -\frac{1}{\sqrt n}X_n (\vt(a)+\d^n(a)).
\end{equation}
A second order Taylor expansion and reordering of the terms yields
\begin{eqnarray}\label{I1}
I_n^{\WW}(a) = &&\underbrace{a\vt(a)-g( \vt(a))}_{{\equiv}I(a)}-\frac{1}{\sqrt n}X_n( \vt(a))\nonumber \\
&& +\underbrace{\left(a- g'( \vt(a))\right)}_{=0}\d^n(a) -\frac{1}{\sqrt n} \d^n(a) X_n'( \vt(a))
\nonumber \\
&&\quad  -\frac 12 ( \d^n(a))^2\left(g''( \vt(a))+\frac{1}{\sqrt n}X_n''( \vt(a))\right) + \po(( \d^n(a))^2).
\end{eqnarray}
Note that the leading terms on the right-hand side involve the three components of 
the processes $\wh X^n$ whose convergence we have just proven.
We obtain the following equation  for $\d^n(a)$ using a first order Taylor expansion.
\begin{eqnarray}\label{bla.11}
a&=&  g'( \vt(a)+\d^n(a)) + \frac{1}{\sqrt n }X_n'(\vt(a)+\d^n(a)) \\ \nonumber 
&= &g'( \vt(a))+\frac{1}{\sqrt n} X_n'( \vt(a))  + \d^n(a) \left(g''(\vt(a)) + \frac{1}{\sqrt n}X_n''( \vt(a))\right) + \po(\d^n(a)),
\end{eqnarray}
which implies 
\begin{equation}\label{bla.12}
 \d^n(a) = \frac{-\frac{1}{\sqrt n}X_n'( \vt(a))}{ g''( \vt(a))+\frac{1}{\sqrt n}X_n''( \vt(a))}+\po(\d^n(a)).
\end{equation}
Lemma \ref{JWC} combined with $g''(\vt(a)) = \OO(1)$ yields $\d^n(a)=\OO(1/\sqrt n)$.
We insert  the expression for $\d^n(a)$ into Equation \eqref{I1}  to obtain 
\begin{eqnarray}\label{bla.13}
&&I_n^{\WW}(a)\nonumber\\
&&=I(a) -\frac{1}{\sqrt n}X_n( \vt(a))
- \frac{ 1}{2}\left(g''( \vt(a))+\frac{1}{\sqrt n}X_n''( \vt(a))\right)\nonumber\\
&&\times \left[ 
\frac{\sfrac 1n(X_n'( \vt(a)))^2}{\left(g''( \vt(a))+\frac{1}{\sqrt n}X_n''( \vt(a))\right)^2}-\frac{\frac{1}{\sqrt n}X_n'( \vt(a))\po(\d^n)}{ \left(g''( \vt(a))+\frac{1}{\sqrt n}X_n''( \vt(a))\right)}+\po((\d^n)^2)\right]\nonumber \\
&&+\frac{\frac{1}{n}(X_n'( \vt(a)))^2}{g''( \vt(a))+\frac{1}{\sqrt n}X_n''( \vt(a))}+ \frac{1}{\sqrt n}X_n'( \vt(a))\po(\d^n)+\po((\d^n)^2).
\end{eqnarray}
Combining this with the bound \eqv(bla.12),  it follows that
\begin{equation}\label{final}
I_n^{\WW}(a)=I(a) -\frac{1}{\sqrt n}X_n( \vt(a)) + \frac 1n r_n(a) ,
\end{equation}
where
\begin{equation}\label{final.2}
r_n(a) \equiv \frac{\frac{1}{2}(X_n'( \vt(a)))^2}{g''( \vt(a))+\frac{1}{\sqrt n}X_n''( \vt(a))} +\po(1).
\end{equation}
 $r_n(a)$ converges weakly due to the  continuous mapping theorem and the joint weak convergence of $X_n'( \vt(a))$ and $X_n''( \vt(a))$. This completes the proof of the theorem.
\end{proof}
 
\bibliographystyle{abbrv}
\bibliography{bib_arxiv}

\end{document}